\documentclass[12pt,a4paper]{article}

\usepackage[margin=2.54cm]{geometry}

\def\Dj{\hbox{D\kern-.73em\raise.30ex\hbox{-}
\raise-.30ex\hbox{}}}
\def\dj{\hbox{d\kern-.33em\raise.80ex\hbox{-}
\raise-.80ex\hbox{\kern-.40em}}}
\usepackage{graphicx}
\usepackage{epstopdf}
\usepackage{amsmath,amsthm,amsfonts,amssymb,amscd,cite}
\allowdisplaybreaks

\usepackage{amsmath, amsthm, amscd, amsfonts, amssymb, graphicx, color}
\usepackage[bookmarksnumbered, plainpages]{hyperref}
\usepackage{pgf,tikz}
\usetikzlibrary{arrows}
\usepackage{diagbox}
\usepackage[margin=1cm,%
            font=small,%
            format=hang,%
            labelsep=period,%
            labelfont=bf]{caption}

\newtheorem{theorem}{Theorem}[section]
\newtheorem{example}{Example}

\newtheorem{lemma}[theorem]{Lemma}

\begin{document}

\baselineskip=0.30in

\vspace*{25mm}

 \begin{center}
 {\Large \bf  On the Number of $k-$Matchings in Graphs}

 \vspace{6mm}

 { \bf  Kinkar Ch. Das$^{1}$,  Ali Ghalavand$^{2,}$\footnote{Corresponding author.} and Ali Reza Ashrafi$^2$}

 \vspace{3mm}

 \baselineskip=0.20in
$^{1}${\it Department of Mathematics, Sungkyunkwan University, Suwon, Republic of Korea\/}\\ 
{\rm E-mail:} {\tt  kinkardas2003@gmail.com}\\
\noindent $^2${\it Department of Pure Mathematics, Faculty of Mathematical Sciences, 
  University of Kashan, Kashan 87317--53153, I. R. Iran\/} \\
 {\rm E-mail:} {\tt alighalavand@grad.kashanu.ac.ir,~ashrafi@kashanu.ac.ir}



 \end{center}

\begin{abstract}\noindent
Suppose $G$ is a undirected simple  graph. A $k-$subset of edges in
$G$ without common vertices is called a $k-$matching and the number
of such subsets is denoted by $p(G,k)$. The aim of this paper is to
present exact formulas for $p(G,3)$, $p(G,4)$ and $p(G,5)$ in terms
of some degree-based invariants.
\vskip 3mm


\vskip 3mm

\end{abstract}

\section{Introduction}
\label{Sec:1}

Let $H$ be a simple and undirected graph with $n = |V(H)|$ vertices
and $m = |E(H)|$ edges. The $n-$vertex path, cycle, complete and
star graphs are denoted by $P_n$, $C_n$, $K_n$ and $S_n$,
respectively. We also use the notation $K_{p,q}$ to denote a
complete bipartite graph with a bipartition $(A,B)$ of $V(H)$ such
that $|A| = p$ and $|B| = q$. The degree of a given vertex $t$ in
$H$, $d_H{(t)}$, is defined to be the number of edges incident to
$t$ and  $N_H(t)$ denotes the set of all vertices adjacent to $t$.
If $W \subseteq V(H)$, then $H - W$ is the subgraph of $H$ with
$V(H)- W$ as its vertex set  and in such a way that all edges in $H$
connecting these vertices. The girth of a graph $H$, $g(H)$, is
defined as the length of a shortest cycle contained in the graph
$H$. If the graph $H$ does not contain any cycles then its girth is
defined to be infinite.

The first and second Zagreb indices of a graph $H$ are two degree
based invariants denoted by $M_1(H)$ and $M_2(H)$, respectively. The
first Zagreb index is defined as the sum of $d_H{(t)}^2$ overall
vertices $t \in V(H)$ and the second Zagreb index is  the sum of
$d_H{(t)}d_H{(s)}$ over all edges $ts \in E(H)$. These degree-based
graph invariants was introduced by  Gutman and Trinajsti\'c in 1972
\cite{6}. They observed the structure-dependency of total
$\pi-$electron energy $\varepsilon(M)$ of a given molecule $M$ with
the first Zagreb index of its molecular graph. We refer to an
interesting paper of Gutman and Das \cite{64} for a history of the
first Zagreb index till 2004. Mili\'cevi\'c et al. \cite{m2}
introduced the first and second reformulated Zagreb indices of a
graph $H$ as edge counterpart of the first and second Zagreb
indices, respectively. These numbers are defined as $EM_1(H)$ $=$
$\sum_{e \sim f}[d_G(e)+d_G(f)]=\sum_{e\in E(G)}d_G(e)^2$ and
$EM_2(G)$ $=$ $\sum_{e \sim f}d_G(e)d_G(f),$ where for $e=uv$,
$d_H(e) = d_H(u) + d_H(v) - 2$ denotes the degree of the edge $e$,
and $e \sim f$ means that the edges $e$ and $f$ are incident.

In the notify paper of Gutman and Trinajsti\'c \cite{6}, the sum of
cubes of degrees of vertices in the molecular graph of $M$ was also
investigated, but this invariant studied in details by Furtula and
Gutman \cite{4}. They named this quantity the forgotten index and
used the notation $F(H)$ to denote the forgotten index of a
molecular graph $H$. In an exact phrase  $F(H)$ $=$ $\sum_{v \in
V(H)}d_H(v)^3$ $=$ $\sum_{e = uv \in E(H)}[d_H(u)^2 + d_H(v)^2]$. A
general form of the first Zagreb index can be defined as
$M_1^\alpha(H)$ $=$ $\sum_{u \in V(H)}d_H(u)^\alpha$, where $\alpha$
is an arbitrary real number except from $0$ and $1$. Zhang and
Zhang \cite{9} obtained some extremal values of this invariant in
the class of all unicyclic graphs of a given order. Notice that the
first Zagreb index   and the forgotten topological index are just
the case of $\alpha = 2,3$ in the first general Zagreb index,
respectively. The interested readers can consult a recent paper of
Gutman et al. \cite{68} for an interesting survey of these modified
Zagreb indices and their main mathematical properties.

Suppose $H$ is a simple undirected graph, $e=xy, f \in E(H)$ and
$v\in V(H)$, where $v\neq x,y$. The common vertex of $e$ and $f$ is
denoted by $e \cap f$ and $e \cap f = \emptyset$ means that $e$ and
$f$ are not incident.  We say that $v$ is incident to $e$ if
$\{vx,vy\}\cap E(H)\neq \emptyset$.  Define:
 \begin{eqnarray*}
 \Lambda(H)&=&\Big\{\{v,e\} \mid v\in V(H),\,e\in E(H) \mbox{ and $v$ is incident to $e$}\Big\},\\
 \alpha_\lambda(H) &=& \sum\limits_{uv \in E(H)}\,d_H(u)\,d_H(v)\,\Big[d_H(u)^\lambda+d_H(v)^\lambda\Big],\\
 \beta(H) &=&  \sum\limits_{e \sim f}\,d_H(e \cap f)\Big(d_H(e)+d_H(f)\Big),\\
 \gamma(H)&=&\displaystyle\sum\limits_{\{v,xy\}\in \Lambda(H)}\,d_H(v)\Big[d_H(x)+d_H(y)\Big],\\
 M_2^\lambda(G)&=& \sum\limits_{uv\in E(H)}\,\Big(d_H(u)d_H(v)\Big)^\lambda,
 \end{eqnarray*}
 where $\lambda$ is an arbitrary real number. We use  $\alpha(H)$,  instead  $\alpha_1(H)$ and the second Zagreb index is just the case of $\lambda = 1$ in $M_2^\lambda(G)$.

Suppose $H$ has exactly $n$ vertices. A $k$-matching is a matching
with $k$ edges and $p(H; k)$ denotes the number of $k$-matchings in
$H$, $k \ne 0$, and by definition $p(H; 0) = 1$. The weight of a
given $k-$matching $M$ is defined as $w_1^{n-2k}w_2^k$ in which
$w_1$ and $w_2$ are weights assigned to each vertex and edge,
respectively. The polynomial
$$p(H)=\sum_{k=0}^{\lfloor\frac{n}{2}\rfloor}p(H;k)w_1^{n-2k}w_2^k$$ is called the
matching polynomial of $H$  and we use the term simple matching
polynomial when $w_1=w_2=w$, see \cite{25} for details. In the
mentioned paper, Farrell presented an algorithm for computing
matching polynomial by which he computed this polynomial for the
path $P_n$, cycle $C_n$, complete graph $K_n$ and complete bipartite
graph $K_{m,n}$. Farrell and Guo \cite{2} continued the study of
matching polynomial and obtained the first four coefficients of
$p(H)$ in terms of some graph invariants related to $H$.

Behmaram \cite{1} obtained a formula for the number of $4-$matchings
in triangular-free graphs in terms of the number of vertices, edges,
degrees and $4-$cycles. He applied this formula to prove that the
Petersen graph is uniquely determined by its matching polynomial.
Fischermann et al. \cite{3} proved sharp upper and lower bounds on
the number of matchings in a tree in terms of the number of
independent and 2-independent sets. Let $T$ be an $n-$vertex tree
and $S(T)$ be the tree obtained from $T$ by replacing each edge of
$T$ by a path of length two. Yan and Yeh \cite{8} proved that the
Wiener index of $T$ can be explained as the number of matchings with
$n - 2$ edges in $S(T)$.

Note that $p(H;1) = |E(H)|$ and if $1\leq k\leq
\lfloor\frac{n}{2}\rfloor$, then
\begin{equation}\label{eq1}
p(H;k)=\frac{1}{k}\sum_{uv\in E(H)}p(H-\{u,v\};k-1).
\end{equation}
By Equation \ref{eq1}, we can obtain formulas for $p(P_n;2), p(S_n;2)$ and $p(C_n;2)$ as follows:

 \begin{eqnarray}
 p(P_n;2)&=&\frac{1}{2}\sum_{uv\in E(P_n)}p(P_n-\{u,v\};1)=\frac{1}{2}\Big[\sum_{i=1}^{n-3}(n-4)+2(n-3)\Big]\nonumber\\
 &=&\frac{1}{2}(n-3)(n-2),~~~\label{aaae1}\\
  p(S_n;2)&=&\frac{1}{2}\sum_{uv\in E(S_n)}p(S_n-\{u,v\};1)=\frac{1}{2}\sum_{uv\in E(S_n)}0=0,\nonumber\\
   p(C_n;2)&=&\frac{1}{2}\sum_{uv\in E(C_n)}p(C_n-\{u,v\};1)=\frac{1}{2}\sum_{uv\in E(C_n)}p(P_{n-2};1)=\frac{1}{2}n(n-3).\nonumber
 \end{eqnarray}
It is merit to state here that this quantities were calculated by Farrell \cite{25}. The aim of this paper is to continue the problem of computing
matching coefficients of a general graph in terms of some graph invariants. Our notations here are standard and can be taken from any standard book in graph theory.

\section{Main Result}
\label{Sec:2}
The aim of this section is to present a formula for the number of $k-$matching, $k \leq 5$, in terms of some degree-based topological indices.

 \begin{lemma}\label{1ml1}
Let $G$ be a graph with $n$ vertices and $m$ edges. Define $\mu_k(G)$ $=$ $\sum_{uv\in V(G)}m(G$ $-\{ u, v\})^k$, for $k=1,2,3,4$. Then
\begin{enumerate}
\item $\mu_1(G)=m^2+m-M_1(G)$.
\item $\mu_2(G)=m^3+F(G)-2mM_1(G)+2M_2(G)+2m^2-2M_1(G)+m$.
\item $\mu_3(G)=m^4-3m^2M_1(G)+3mF(G)+6mM_2(G)-M_1^4(G)-3\alpha(G)+3m^3
-6mM_1(G)+3F(G)+6M_2(G)+3m^2-3M_1(G)+m$.
\item $\mu_4(G)=M_1^5(G)+4\alpha_2(G)-4mM_1^4(G)+6M_2^2(G)-12m\alpha(G)+6m^2F(G)
+12m^2M_2(G)$ $-4m^3M_1(G)-4M_1^4(G)-12\alpha(G)+12mF(G)
+24mM_2(G)-12m^2M_1(G)+6F(G)+12M_2(G)-12mM_1(G)-4M_1(G)
+m^5+4m^4+6m^3+4m^2+m$.
\end{enumerate}
 \end{lemma}

\begin{proof}
By definition, $\mu_k(G)$ $=$ $\sum\limits_{uv\in
E(G)}\Big[m-\big(d_G(u)+d_G(v)-1\big)\Big]^k$ and  a simple
calculation gives the result.
\end{proof}

 Let $d_G(u)$ be the degree of the vertex $u$ in $G$. Also let $m_G(u)$ be the average degree of the adjacent vertices of vertex
 $u$ in $G$. Then
   $$m_G(u)=\frac{\sum\limits_{v:vu\in E(G)}\,d_G(v)}{d_G(u)},~\mbox{ that is,}~\sum\limits_{v:vu\in E(G)}\,d_G(v)=d_G(u)\,m_G(u).$$
 One can easily see that
 \begin{equation}
 \sum\limits_{u\in V(G)}\,d_G(u)\,m_G(u)=\sum\limits_{u\in V(G)}\,\sum\limits_{v:vu\in E(G)}\,d_G(v)=\sum\limits_{u\in V(G)}\,d_G(u)^2=M_1(G)\nonumber
 \end{equation}
 and
 \begin{eqnarray}
 \sum_{uv\in E(G)}\,\Big[d_G(u)\,m_G(u)+d_G(v)\,m_G(v)\Big]&=&\sum_{u\in V(G)}\,d_G(u)^2\,m_G(u)\nonumber\\
                                                      &=&\sum_{u\in V(G)}\,d_G(u)\,\sum_{v:vu\in E(G)}\,d_G(v)\nonumber\\
                                                      &=&2\,\sum_{uv\in E(G)}\,d_G(u)\,d_G(v)=2\,M_2(G).\label{1kan2}
 \end{eqnarray}

\begin{lemma}\label{l1}
Let $G$ be a graph with $n$ vertices, $m$ edges and $g(G) \geq 4$. For $k=2,3,4$, we define $ \xi_{k-1}(G)=\sum_{uv\in V(G)}M_1^k(G-\{ u, v\})$. Then the following hold:
\begin{enumerate}
\item $\xi_1(G)=(m+3)\,M_1(G)-F(G)-4M_2(G)-2m $.
\item $\xi_2(G)=(m+3)\,F(G)-M_1^4(G)-3\alpha(G)+6M_2(G)-4M_1(G)+2m$.
\item $\xi_3(G)=(m+4)\,M_1^4(G)-M_1^5(G)+5M_1(G)-2m-4\alpha_2(G)+6\alpha(G)-6F(G)-8M_2(G)$.
\end{enumerate}
\end{lemma}

\begin{proof} By definition of graph invariants with (\ref{1kan2}), one can see that

\begin{small}
 \begin{flalign*}
&(1)\hspace*{3mm}\xi_1(G)\\
&=\sum_{uv\in E(G)}\Bigg[ M_1(G)-\Bigg(d_G(u)^2+d_G(v)^2+\sum_{x:xu\in E(G)\atop x\neq v}\Big(2d_G(x)-1\Big)+\sum_{y:yv\in E(G)\atop y\neq u}\Big(2d_G(y)-1\big)\Bigg)\Bigg]\\
&=mM_1(G)-F(G)-\sum_{uv\in E(G)}\Bigg[\sum_{x:xu\in E(G)\atop x\neq v}\Big(2d_G(x)-1\Big)+\sum_{y:yv\in E(G)\atop y\neq u}\Big(2d_G(y)-1\Big)\Bigg]\\
&=mM_1(G)-F(G)-\sum_{uv\in
E(G)}\Bigg[2\Big(d_G(u)\,m_G(u)+d_G(v)\,m_G(v)\Big)-3\Big(d_G(u)+d_G(v)\Big)+2\Bigg]\\[2mm]
&=(m+3)\,M_1(G)-F(G)-4M_2(G)-2m.
 \end{flalign*}
 \end{small}
 Since $g(G)\geq 4$, by definition of graph invariants, one can see that
\begin{eqnarray*}
(2)~\xi_2(G)&=&\sum_{uv\in E(G)}\Bigg[ F(G)-\Bigg(d_G(u)^3+d_G(v)^3+\sum_{x:xu\in E(G)\atop x\neq v}\Big(3d_G(x)^2-3d_G(x)+1\Big)\\
&&\hspace*{5.6cm}~+\sum_{y:yv\in E(G)\atop y\neq
u}\Big(3d_G(y)^2-3d_G(y)+1\Big)\Bigg)\Bigg]\\
&=&mF(G)-M_1^4(G)-\sum_{uv\in E(G)}\Bigg[\sum_{x:xu\in E(G)\atop x\neq v}\Big(3d_G(x)^2-3d_G(x)+1\Big)\\
&&\hspace*{5.6cm}+\sum_{y:yv\in E(G)\atop y\neq
u}\Big(3d_G(y)^2-3d_G(y)+1\Big) \Bigg]\\
 &=&mF(G)-M_1^4(G)-\sum_{uv\in E(G)}\Bigg[\Big(3d_G(u)^2-3d_G(u)+1\Big)\Big(d_G(v)-1\Big)\\
&&\hspace*{5.6cm}+\Big(3d_G(v)^2-3d_G(v)+1\Big)\Big(d_G(u)-1\Big)\Bigg]\\
 &=&mF(G)-M_1^4(G)-\sum_{uv\in E(G)}\Bigg(3\Big(d_G(u)^2d_G(v)+d_G(u)d_G(v)^2\Big)\\
&&\hspace*{1.6cm}-3\Big(d_G(u)^2+d_G(v)^2\Big)-6d_G(u)d_G(v)+4\Big(d_G(u)+d_G(v)\Big)-2\Bigg)\\
 &=&(m+3)\,F(G)-M_1^4(G)-3\alpha(G)+6M_2(G)-4M_1(G)+2m,
 \end{eqnarray*}
 
 \begin{small}
 \begin{eqnarray*}
&(3)&~\xi_3(G)\\
&=&mM_1^4(G)-\sum_{v\in V(G)}d_G(v)^5-\sum_{uv\in E(G)}\Bigg[\sum_{x:xu\in E(G)\atop x\neq v}\Big(4d_G(x)^3-6d_G(x)^2+4d_G(x)-1\Big)\\
&&\hspace*{3cm}+\sum_{y:yv\in E(G)\atop y\neq u}\Big(4d_G(y)^3-6d_G(y)^2+4d_G(y)-1\Big)\Bigg]\\
&=&mM_1^4(G)-M_1^5(G)-\sum_{uv\in E(G)}\Bigg[\Big(4d_G(v)^3-6d_G(v)^2+4d_G(v)-1\Big)\\
&&\hspace*{2cm}\times\Big(d_G(u)-1\Big)+\Big(4d_G(u)^3-6d_G(u)^2+4d_G(u)-1\Big)\Big(d_G(v)-1\Big)\Bigg]\\
&=&(m+4)\,M_1^4(G)-M_1^5(G)+5M_1(G)-2m-4\alpha_2(G)+6\alpha(G)-6F(G)-8M_2(G).
 \end{eqnarray*}
  \end{small}
This completes the proof of the lemma.
\end{proof}

\begin{lemma}\label{cl1}
Let $G$ be a graph with $n$ vertices, $m$ edges and $g(G)\geq4$. Set
$\varrho_1(G)$ $=$ $\sum_{uv\in V(G)}m(G-\{ u, v\})M_1(G-\{ u,
v\})$. Then,
\begin{eqnarray*}
\varrho_1(G)&=&(m^2+4m+5)\,M_1(G)-(m+2)\,F(G)-(4m+6)\,M_2(G)+M_1^4(G) - M_1(G)^2\\
&&+\alpha(G)-2m^2-2m+2\gamma(G).
\end{eqnarray*}
\end{lemma}

\begin{proof} Since $g(G)\geq4$, by definition, we obtain
 \begin{eqnarray}
 \varrho_1(G)&=&\sum_{uv\in E(G)}\Big[m-\Big(d_G(u)+d_G(v)-1\Big)\Big]\cdot\Big[ M_1(G)-\Big(d_G(u)^2+d_G(v)^2\nonumber\\
 &&\hspace*{4cm}+\sum_{x:xu\in E(G)\atop x\neq v}(2d_G(x)-1)+\sum_{y:yv\in E(G)\atop y\neq u}(2d_G(y)-1)\Big)\Big]\nonumber\\
 &=&(m+1)\,m\,M_1(G)-(m+1)\,F(G)-(m+1)\Big[4M_2(G)-3M_1(G)+2m\Big]\nonumber\\
 &&-M_1(G)^2+\sum\limits_{uv\in E(G)}\,\Big(d_G(u)+d_G(v)\Big)\,\Big(d_G(u)^2+d_G(v)^2\Big)\nonumber\\
 &&+\sum\limits_{uv\in E(G)}\,\Big(d_G(u)+d_G(v)\Big)\Big[\sum_{x:xu\in E(G)\atop x\neq v}(2d_G(x)-1)+\sum_{y:yv\in E(G)\atop y\neq u}(2d_G(y)-1)\Big].~~\label{1kin1}
 \end{eqnarray}
 Since
 $$\sum\limits_{uv\in E(G)}\,\Big(d_G(u)+d_G(v)\Big)\,\Big(d_G(u)^2+d_G(v)^2\Big)=M^4_1(G)+\alpha(G),$$
 and
 \begin{eqnarray*}
 &&\sum\limits_{uv\in E(G)}\,\Big(d_G(u)+d_G(v)\Big)\Big[\sum_{x:xu\in E(G)\atop x\neq v}(2d_G(x)-1)+\sum_{y:yv\in E(G)\atop y\neq u}(2d_G(y)-1)\Big]\\
 &=&2\,\sum\limits_{uv\in E(G)}\,\Big(d_G(u)+d_G(v)\Big)\Big[\sum_{x:xu\in E(G)\atop x\neq v}\,d_G(x)+\sum_{y:yv\in E(G)\atop y\neq u}\,d_G(y)\Big]\\
 &&-\sum\limits_{uv\in E(G)}\,\Big(d_G(u)+d_G(v)\Big)^2+2\sum\limits_{uv\in E(G)}\,\Big(d_G(u)+d_G(v)\Big)\\
 &=&2\gamma(G)-F(G)-2\,M_2(G)+2\,M_1(G),
 \end{eqnarray*}
 from (\ref{1kin1}), we get the required result.
 \end{proof}

\begin{lemma}\label{al1}
{\rm (See \cite{0} )} Let $G$ be a graph. Then
$\beta(G)-\alpha(G)=M_1^4(G)-3F(G)+2M_1(G)-2M_2(G)$.
\end{lemma}

 \begin{lemma}\label{al2} Let $G$ be a graph with $m$ edges and $g(G)\geq 5$. If suppose that, $\eta_1(G)=\sum_{uv\in V(G)}M_2(G-\{ u, v\})$, then
 \begin{eqnarray*}
 &&\eta_1(G)=mM_2(G)-2\,EM_2(G)+\beta(G)-2F(G)-7M_2(G)+9M_1(G)-8m.
 \end{eqnarray*}
 \end{lemma}

 \begin{proof} One can easily see that
 \begin{eqnarray*}
 &&\sum_{uv\in E(G)}\,\Bigg[\sum_{x:xu\in E(G)}\,d_G(u)d_G(x)+\sum_{y:yv\in E(G)}\,d_G(v)d_G(y)\Bigg]\\
 &=&\sum_{uv\in E(G)}\,\Big[d_G(u)^2\,m_G(u)+d_G(v)^2\,m_G(v)\Big]\\
 &=&\sum_{u\in E(G)}\,d_G(u)^3\,m_G(u)\\
 &=&\sum_{u\in E(G)}\,d_G(u)^2\,\sum_{v:vu\in E(G)}\,d_G(v)\\
 &=&\sum_{uv\in E(G)}d_G(u)d_G(v)\Big[d_G(u)+d_G(v)\Big]=\alpha(G).
 \end{eqnarray*}

 \noindent
 If $u\sim v\sim w$ is a induced path graph in $G$ with $g(G)\geq 5$, then we have
 \begin{eqnarray*}
 &&\sum_{u\in V(G)}\,d_G(u)\,\sum_{v:vu\in E(G)}\,d_G(v)\,m_G(v)\\[2mm]
        &=&\sum_{u\in V(G)}\,d_G(u)\,\sum_{v:vu\in E(G)}\,\sum_{w:wv\in E(G)}\,d_G(w)\\[2mm]
        &=&2\,\sum_{u\sim v\sim w}\,d_G(u)\,d_G(w)+\sum_{u\in V(G)}\,d_G(u)\,\sum_{v:vu\in E(G)}\,d_G(u)\\[2mm]
        &=&2\,\sum_{uv=e\sim f=vw}\,\Big(d_G(e)-d_G(v)+2\Big)\,\Big(d_G(f)-d_G(v)+2\Big)+\sum_{u\in V(G)}\,d_G(u)^3\\[2mm]
        &=&2\,EM_2(G)+4\,EM_1(G)-2\,\beta(G)+M^4_1(G)-4F(G)+8M_1(G)-8m.
 \end{eqnarray*}

 \noindent
 Now,
 \begin{eqnarray*}
  &&\sum_{uv\in E(G)}\,\Bigg[\sum_{x:xu\in E(G),\atop x\neq v}\,\sum_{w:wx\in E(G),\atop w\neq u}\,d_G(w)+\sum_{y:yv\in E(G),\atop y\neq u}\,\sum_{w:wy\in E(G),\atop w\neq v}\,d_G(w)\Bigg]\\[2mm]
  &=&\sum_{uv\in E(G)}\,\Bigg[\sum_{x:xu\in E(G),\atop x\neq v}\,\Big(d_G(x)\,m_G(x)-d_G(u)\Big)+\sum_{y:yv\in E(G),\atop y\neq u}\,\Big(d_G(y)\,m_G(y)-d_G(v)\Big)\Bigg]\\[2mm]
  &=&\sum_{uv\in E(G)}\,\Bigg[\sum_{x:xu\in E(G)}\,\Big(d_G(x)\,m_G(x)-d_G(u)\Big)+\sum_{y:yv\in E(G)}\,\Big(d_G(y)\,m_G(y)-d_G(v)\Big)\\[2mm]
 &&\hspace*{5cm}-\Big(d_G(v)\,m_G(v)-d_G(u)+d_G(u)\,m_G(u)-d_G(v)\Big)\Bigg]\\[2mm]
 &=&\sum_{uv\in E(G)}\,\Bigg[\sum_{x:xu\in E(G)}\,d_G(x)\,m_G(x)+\sum_{y:yv\in E(G)}\,d_G(y)\,m_G(y)-d_G(u)^2-d_G(v)^2\Bigg]\\[2mm]
 &&\hspace*{4cm}-\sum_{uv\in E(G)}\,\Bigg[d_G(u)\,m_G(u)+d_G(v)\,m_G(v)-d_G(u)-d_G(v)\Bigg]\\[2mm]
 &=&\sum_{u\in V(G)}\,d_G(u)\,\sum_{v:vu\in E(G)}\,d_G(v)\,m_G(v)-F(G)-2\,M_2(G)+M_1(G)\\[2mm]
 &=&2\,EM_2(G)+4\,EM_1(G)-2\,\beta(G)+M^4_1(G)-5F(G)+9M_1(G)-2\,M_2(G)-8m.
 \end{eqnarray*}

 \vspace*{3mm}

 \noindent
 Using the above results, we obtain
 \begin{eqnarray*}
 \eta_1(G)&=&\sum_{uv\in E(G)}\,\Bigg[M_2(G)+d_G(u)d_G(v)-\sum_{x:xu\in E(G)}\,d_G(u)d_G(x)-\sum_{y:yv\in E(G)}\,d_G(v)d_G(y)\\[2mm]
 &&\hspace*{4cm}-\sum_{x:xu\in E(G),\atop x\neq v}\,\sum_{w:wx\in E(G),\atop w\neq u}\,d_G(w)-\sum_{y:yv\in E(G),\atop y\neq u}\,\sum_{w:wy\in E(G),\atop w\neq v}\,d_G(w)\Bigg]
 \end{eqnarray*}
 \begin{eqnarray*}
 &=&m\,M_2(G)-\sum_{uv\in E(G)}\,\Bigg[\sum_{x:xu\in E(G)}\,d_G(u)d_G(x)+\sum_{y:yv\in E(G)}\,d_G(v)d_G(y)-d_G(u)d_G(v)\Bigg]\\[2mm]
 &&\hspace*{2cm}-\sum_{uv\in E(G)}\,\Bigg[\sum_{x:xu\in E(G),\atop x\neq v}\,\sum_{w:wx\in E(G),\atop w\neq u}\,d_G(w)+\sum_{y:yv\in E(G),\atop y\neq u}\,\sum_{w:wy\in E(G),\atop w\neq v}\,d_G(w)\Bigg]\\[2mm]
 &=&mM_2(G)-\alpha(G)-2\,EM_2(G)-4\,EM_1(G)+2\,\beta(G)-M^4_1(G)+5F(G)-9M_1(G)\\[2mm]
 &&\hspace*{6cm}+3\,M_2(G)+8m.
 \end{eqnarray*}
 Moreover,
 \begin{equation*}\label{equa1}
 EM_1(G)=\sum_{uv=e\in E(G)}\,\Big(d_G(u)+d_G(v)-2\Big)^2=F(G)+2M_2(G)-4M_1(G)+4m.
 \end{equation*}
 Using the above results with Lemma \ref{al1}, we get the required result.
 \end{proof}

\begin{lemma}\label{al3}
Let  $G$ be a graph with $m$ edges. If $g(G)>3$, then $\gamma(G)$
$=$ $2EM_2(G)$ $-$ $\beta(G)$ $+$ $4EM_1(G)$ $-$ $2F(G)+6M_1(G)-8m$.
\end{lemma}

 \begin{proof} By definition,
 \begin{small}
 \begin{eqnarray*}
\gamma(G)&=&\sum_{\{v,xy\}\in \Lambda(G)}d_G(v)\Big[d_G(x)+d_G(y)\Big]\\
&=&\sum_{e=uv\sim f=
vw}\,\Big[\Big(d_G(e)-d_G(v)+2\Big)\Big(d_G(f)+2\Big)+\Big(d_G(f)-d_G(v)+2\Big)\Big(d_G(e)+2\Big)\Big]\\[2mm]
&=&2EM_2(G)-\beta(G)+4EM_1(G)-2F(G)+6M_1(G)-8m,
\end{eqnarray*}
 \end{small}
 as desired.
\end{proof}

\begin{theorem}\label{1rl1}
{\rm (See \cite{2} )}
Let $G$ be a graph with $m$ edges. Then
\begin{eqnarray*}
p(G,2)&=&  \frac{1}{2}m^2+\frac{1}{2}m-\frac{1}{2}M_1(G).
\end{eqnarray*}
\end{theorem}

 \begin{proof}
 By Equation \ref{eq1} and Lemma \ref{1ml1},
 \begin{eqnarray*}
  p(G;2)=\frac{1}{2}\sum_{uv\in E(G)}p(G-\{u,v\};1)&=&\frac{1}{2}\sum_{uv\in E(G)}m(G-\{u,v\})\\
  &=&\frac{1}{2}\mu_1(G)=\frac{1}{2}m^2+\frac{1}{2}m-\frac{1}{2}M_1(G),
 \end{eqnarray*}
 as desired.
 \end{proof}

\begin{theorem}\label{1rl2}
{\rm (See \cite{2} )} Let $G$ be a graph with $m$ edges and
$g(G)>3$. Then
\begin{eqnarray*}
p(G,3)&=&\frac{1}{6}m^3+\frac{1}{2}m^2+\frac{2}{3}m-\frac{1}{2}mM_1(G)-M_1(G)+\frac{1}{3}F(G)+M_2(G).
\end{eqnarray*}
\end{theorem}

\begin{proof}
Apply Equation \ref{eq1} and  Theorem \ref{1rl1}, we obtain
$p(G;3)=\frac{1}{6}\mu_2(G)+\frac{1}{6}\mu_1(G)-\frac{1}{6}\xi_1(G).$
Now the theorem follows from Lemmas \ref{1ml1} and  \ref{l1}.
\end{proof}

\begin{theorem}\label{1rl3}
Let $G$ be a graph with $m$ edges and $g(G)\geq 5$. Then
\begin{small}
\begin{eqnarray*}
p(G,4)&=&\frac{1}{24}m^4+\frac{1}{4}m^3+\frac{19}{24}m^2-\frac{11}{4}m
 +\frac{1}{8}(M_1(G))^2+\frac{1}{3}mF(G)-\frac{1}{4}m^2M_1(G)+\\
 &&mM_2(G)+\frac{1}{4}M_1^4(G)
-2M_2(G) -\frac{5}{4}mM_1(G)+\frac{7}{2}
M_1(G)-EM_2(G)-\frac{3}{2}F(G).
\end{eqnarray*}
\end{small}
\end{theorem}

\begin{proof}
By Equation \ref{eq1} and Theorem \ref{1rl2}, we obtain
\begin{eqnarray*}
 p(G;4) &=&\frac{1}{24}\mu_3(G)+\frac{1}{8}\mu_2(G)+\frac{1}{6}\mu_1(G)-\frac{1}{8}\varrho_1(G)-\frac{1}{4}\xi_1(G)
 +\frac{1}{12}\xi_2(G)+\frac{1}{4}\eta_1(G).
 \end{eqnarray*}
Now, by Lemmas  \ref{1ml1}--\ref{al3} and Equation \ref{equa1}, that
gives the result.
\end{proof}

\begin{theorem}\label{aat1}
Let $G$ be a graph with $m$ edges and $g(G)\geq 5$. Then
\begin{small}
\begin{eqnarray*}
 &&\hspace*{-7mm}p(G;5)\\
 &=&\frac{1}{5}\Bigg[\frac{1}{24}m(m^4+10m^3+43m^2+54m-328)+\frac{5}{4}(M_1(G))^2-\frac{1}{2}\alpha(G)(m-7)
 -\frac{5}{6}\alpha_2(G)\\[2mm]
&&-\frac{1}{12}M_1(G)(2m^3+30m^2+61m-225)+\frac{1}{2}\beta(G)+\frac{1}{12}M_2(G)(6m^2+66m-239)\\[2mm]
&&+\frac{1}{24}F(G)(6m^2+24m-149)+\frac{1}{12}M_1^4(G)(m+10)+\frac{1}{4}M_2^2(G)-EM_2(G)-\frac{5}{24}M_1^5(G)\\[2mm]
&&+\frac{1}{8}\sum_{uv\in E(G)}(M_1(G-\{u,v\}))^2+\frac{1}{3}\sum_{uv\in E(G)}m(G-\{u,v\})F(G-\{u,v\})\\[2mm]
 &&-\frac{1}{4}\sum_{uv\in E(G)}m^2(G-\{u,v\})M_1(G-\{u,v\})-\sum_{uv\in E(G)}EM_2(G-\{u,v\})\\[2mm]
 &&+\sum_{uv\in E(G)}m(G-\{u,v\})M_2(G-\{u,v\})  \Bigg].
\end{eqnarray*}
\end{small}
\end{theorem}

\begin{proof} By Equation \ref{eq1} and Theorem \ref{1rl3}, we obtain

\begin{small}
\begin{eqnarray*}
 &&\hspace*{-7mm}p(G;5)\\
&=&\frac{1}{5}\Big[\frac{1}{24}\mu_4(G) +\frac{1}{4}\mu_3(G)+\frac{19}{24}\mu_2(G)-\frac{11}{4}\mu_1(G)
 +\frac{1}{8}\sum_{uv\in E(G)}(M_1(G-\{u,v\}))^2\\[2mm]
 &&+\frac{1}{3}\sum_{uv\in E(G)}m(G-\{u,v\})F(G-\{u,v\})-\frac{1}{4}\sum_{uv\in E(G)}m^2(G-\{u,v\})M_1(G-\{u,v\})\\[2mm]
 &&+\sum_{uv\in E(G)}m(G-\{u,v\})M_2(G-\{u,v\})+\frac{1}{4}\xi_3(G)
-2\eta_1(G) -\frac{5}{4}\varrho_1(G)\\[2mm]
&&+\frac{7}{2}\xi_1(G)-\sum_{uv\in
E(G)}EM_2(G-\{u,v\})-\frac{3}{2}\xi_2(G)\Big],
\end{eqnarray*}
\end{small}
Now,    Lemmas \ref{1ml1}--\ref{al3}, and Equation \ref{equa1} give
the result.
\end{proof}
\section{Applications}

In this section, we apply our results to calculate $p(G,k)$, $k = 3,4,5,6$ for some known graphs.
\begin{example}\label{aex1}
In this example, we consider the path $P_n$ with $n$ vertices and $m
= n-1$ edges. A simple calculation shows that $M_1(P_n)$ $=$ $4n-6$,
$F(G) = 8n-14$, $M_1^4(G)= 16n-30$, $M_1^5(P_n)= 32n-62$, $M_2(P_n)=
4n-8$, $\alpha(P_n)= 16n-36$, $\alpha_2(P_n)$ $=$ $32n-76$,
$M_2^2(P_n)= 16n-40$, $\beta(P_n)= 8n-20$ and $EM_2(P_n)= 4n-12$.
Furthermore, $\sum_{uv\in E(P_n)}M_1(P_n-\{u,v\})^2$ $=$
$16(n-4)(n^2-7n+15)$, $\sum_{uv\in
E(P_n)}m(P_n-\{u,v\})F(P_n-\{u,v\})$ $=$ $8n^3-84n^2+308n-396,$
$\sum_{uv\in E(P_n)}EM_2(P_n-\{u,v\})$ $=$ $4n^2-36n+82$,
$\sum_{uv\in E(P_n)}m^2(P_n-\{u,v\})M_1(P_n-\{u,v\})$ $=$
$4n^4-56n^3+308n^2-784n+772$ and $\sum_{uv\in
E(P_n)}m(P_n-\{u,v\})M_2(P_n-\{u,v\})$ $=$ $2(n-4)(2n^2-14n+29)$.
Therefore, by Theorems \ref{1rl2}--\ref{aat1}, we obtain
\begin{eqnarray*}
p(P_n;3)&=&\frac{1}{6}(n-3)(n-4)(n-5),\\
p(P_n;4)&=&\frac{1}{24}(n-4)(n-5)(n-6)(n-7),\\
p(P_n;5)&=&\frac{1}{120}(n-5)(n-6)(n-7)(n-8)(n-9).
\end{eqnarray*}
\end{example}
In general, we obtain the following result. For this let $G$ be a
graph with given edge $uv$. Then $p(G,uv;k)$ denotes the number of
$k$-matchings in $G$ that contain the edge $uv$.
\begin{example} \label{ak1} Let $P_n$ be a path with $n$ vertices and a positive integer $k\leq n/2$. Then
$$p(P_n;k)={n-k \choose k}.$$
We prove this result by mathematical induction on $n$. For $n=2$,
$p(P_2;1)={1 \choose 1}=1$, and the result holds. Suppose that $n>2$
and $P_n:\,v_1v_2v_3\ldots v_n$. Then by definition and induction
hypothesis we have,
\begin{eqnarray*}
p(P_n;k)&=&p(P_n,v_1v_2;k)+p(P_n,v_2v_3;k)+p(P_{n-2};k)\\[2mm]
&=&p(P_{n-2};k-1)+p(P_{n-3};k-1)+p(P_{n-2};k)\\[2mm]
&=&{n-k-1\choose k-1}+{n-k-2\choose k-1}+{n-k-2\choose k}={n-k
\choose k}.
\end{eqnarray*}
\end{example}

\begin{example}
Consider the $n-$cycle $C_n$. Apply Equation \ref{eq1} and Lemma
\ref{ak1}, we obtain
$$p(C_n;k)=\frac{1}{k}\sum_{uv\in E(C_n)}\,p(C_{n}-\{u,v\};k-1)=\frac{n}{k}\,p(P_{n-2};k-1)=\frac{n}{k}{n-k-1 \choose k-1}.$$
By applying this equality, one can easily see that
\begin{eqnarray*}
p(C_n;3)&=&\frac{n}{6}(n-4)(n-5),\\
p(C_n;4)&=&\frac{n}{24}(n-5)(n-6)(n-7),\\
p(C_n;5)&=&\frac{n}{120}(n-6)(n-7)(n-8)(n-9),\\
p(C_n;6)&=&\frac{n}{720}(n-7)(n-8)(n-9)(n-10)(n-11).
\end{eqnarray*}
\end{example}

\begin{example}\label{aex2}
Let the vertices of the path $P_{k}$ be numbered consecutively by
$1,2,\ldots,k$. Construct the graph $P_{k,k-4}$ by attaching a
pendent vertex at positions $3$ to $k-2$ of the $k$-vertex path. By
simple calculations, Theorems \ref{1rl2}--\ref{aat1} we have
\begin{eqnarray*}
p(P_{k,k-4};2)&=&2k(k-7)+25~\mbox{ for}~k\geq 4,\\
p(P_{k,k-4};3)&=&\frac{1}{3}(2k-9)(2k^2-18k+43)~\mbox{ for}~k\geq 5,\\
p(P_{k,k-4};4)&=&\frac{2}{3}k(k-11)(k^2-11k+64)+681~\mbox{ for}~k\geq 6,\\
p(P_{k,k-4};5)&=&\frac{1}{15}(2k-13)(2k^4-52k^3+522k^2-2392k+4215)~\mbox{
for}~k\geq 7.
\end{eqnarray*}
\end{example}

\begin{example}
Let the vertices of the cycle $C_{k}$ be numbered consecutively by $1,2,\ldots,k$. Construct
the graph $C_{k,k}$ by attaching a pendent vertex at positions $1$ to $k$ of the $k$-vertex cycle. By Equation \ref{eq1},\\
\begin{eqnarray*}
p(C_{k,k};r)&=&\frac{1}{r}\Big[\sum_{i=1}^kp(P_{k+1,k-3};r-1)+\sum_{i=1}^kp(P_{k,k-4};r-1)\Big]\\
&=&\frac{k}{r}\Big[p(P_{k+1,k-3};r-1)+p(P_{k,k-4};r-1)\Big].
\end{eqnarray*}
 Now, by  Example \ref{aex2},
\begin{eqnarray*}
p(C_{k,k-4};3)&=&\frac{2}{3}k(2k^2-12k+19)~\mbox{ for}~k\geq 4,\\
p(C_{k,k-4};4)&=&\frac{2}{3},k(k-4)(k^2-8k+18)~\mbox{ for}~k\geq 5,\\
p(C_{k,k-4};5)&=&\frac{2}{15}k(2k^4-40k^3+310k^2-1100k+1503)~\mbox{ for}~k\geq 6,\\
p(C_{k,k-4};6)&=&\frac{2}{45}k(k-6)(2k^4-48k^3+452k^2-1968k+3335)~\mbox{
for}~k\geq 7.
\end{eqnarray*}
\end{example}

\section*{Acknowledgment}
K. C. Das is supported by the National Research Foundation of the Korean government with grant No.
 2017R1D1A1B03028642. A. Ghalavand and A. R. Ashrafi are partially supported by the University of Kashan under grant number 890190/2.

\end{document}